\documentclass[11pt]{amsart}

\usepackage{amsthm,amsmath,amssymb}
\usepackage[utf8]{inputenc} 
\usepackage[english]{babel}
\usepackage{indentfirst}
\usepackage[T1]{fontenc}

\usepackage{enumitem}
\usepackage{graphicx}
\usepackage{hyperref}

\theoremstyle{definition}
\newtheorem{definition}{Definition}

\theoremstyle{plain}
\newtheorem{theorem}{Theorem}
\newtheorem{lemma}{Lemma}
\newtheorem{prop}{Proposition}

\newtheorem*{conjecture*}{Conjecture}

\newtheorem*{fact*}{Fact}

\theoremstyle{remark}
\newtheorem{remark}{Remark}

\newcommand{\abs}[1]{\left\lvert #1 \right\rvert}
\newcommand{\norm}[1]{\lVert #1 \rVert}
\newcommand{\set}[1]{\left\{\, #1 \,\right\}}
\newcommand{\bracket}[1]{\langle #1 \rangle}

\newcommand{\Real}{\mathbf{R}}
\newcommand{\Rd}{\mathbf{R}^d}

\newcommand{\Integ}{\mathbf{Z}}
\newcommand{\Tor}{\mathbf{T}}
\newcommand{\Sph}{\mathbf{S}}

\newcommand{\imply}{\Rightarrow}
\newcommand{\conv}{\star}

\newcommand{\abss}[1]{\lvert #1 \rvert}
\newcommand{\eps}{\varepsilon}

\let\div\relax
\DeclareMathOperator{\div}{div}

\DeclareSymbolFont{bbold}{U}{bbold}{m}{n}
\DeclareSymbolFontAlphabet{\mathbbold}{bbold}
\newcommand{\ind}{\mathbbold{1}}

\def\XXint#1#2#3{{\setbox0=\hbox{$#1{#2#3}{\int}$ }
\vcenter{\hbox{$#2#3$ }}\kern-.6\wd0}}

\newcommand{\chieps}{\chi_{\eps}}

\newcommand{\normlinf}[1]{\norm{#1}_{L^{\infty}}}
\newcommand{\normlone}[1]{\norm{#1}_{L^1}}
\newcommand{\normltwo}[1]{\norm{#1}_{L^2}}
\newcommand{\normlp}[1]{\norm{#1}_{L^p}}
\newcommand{\normhms}[1]{\norm{#1}_{\dot H^{-s}}}

\newcommand{\normhmo}[1]{\norm{#1}_{\dot H^{-1}}}

\newcommand{\calW}{\mathcal{W}}
\newcommand{\calV}{\mathcal{V}}

\newcommand{\FT}{\mathcal{F}}

\newcommand{\Schwartz}{\mathcal{S}}

\renewcommand{\th}{\theta}

\newcommand{\om}{\omega}

\title{A new approach to bounds on mixing}
\author{Flavien Léger}

\begin{document}
\maketitle

\begin{abstract}
We consider mixing by incompressible flows. In 2003, Bressan stated a conjecture concerning a bound on the mixing achieved by the flow in terms of an $L^1$ norm of the velocity field. Existing results in the literature use an $L^p$ norm with $p>1$. In this paper we introduce a new approach to prove such results. It recovers most of the existing results and offers new perspective on the problem. Our approach makes use of a recent harmonic analysis estimate from Seeger, Smart and Street. 
\end{abstract}

\section{Introduction}

Consider a passive scalar $\th$ which is advected by a time-dependent and divergence-free velocity field $u$ on the whole space $\Rd$. $\th$ might assign labels to the fluid particles, or represent the concentration of some scalar quantity. $\th$ and $u$ then satisfy the equations
\begin{align}
\begin{cases}\label{eq:conservation_of_mass}
\partial_t\th+\div(u\th) = 0\\
\div(u) = 0 & \text{on } [0,+\infty)\times\Rd\\
\th(0,\cdot) = \th_0 & \text{on } \Rd
\end{cases}
\end{align}

In~\cite{bressan}, Bressan stated a conjecture relating a bound on the mixing achieved by the flow in terms of an $L^1$ norm of the velocity field:
\begin{conjecture*}[Bressan,~\cite{bressan}]
For the geometric mixing scale $\eps(t)$ of $\th(t,\cdot)$ (see def.~\ref{def:geometric_mixing_scale} in Section~\ref{sec:mixing_def}) there exists a constant $C>0$ depending on $\th_0$ such that
\begin{equation*}
\eps(t)\ge C^{-1} \exp\left(-C\int_0^t\norm{\nabla u(t',\cdot)}_1 dt'\right)
\end{equation*}
\end{conjecture*}

To the best of our knowledge this conjecture is still open. However, starting from work by Crippa and De Lellis~\cite{crippadelellis}, there have been several related results bounding mixing in terms of the $L^p$ norm in space $\int_0^t\norm{\nabla u(t',\cdot)}_p dt'$ with $p>1$. 

This paper develops a new approach to proving bounds on mixing. It recovers many of the known results, and (being different from the previous proofs) offers new perspective on the problem. Unfortunately, this method doesn't seem to be effective to deal with the $L^1$ case. 

To describe our results, we begin by introducing the following functional (this is different from, but somewhat analogous to, the functionals discussed in~\cite{crippadelellis},\cite{jabin},\cite{seeger},\cite{seegermixing})
\begin{equation*}
\calV(f) = \int_{\Rd} \log{\abs{\xi}}\,\abss{\hat{f}(\xi)}^2\,d\xi
\end{equation*}
where $\hat{f}$ denotes the Fourier transform of $f$. $\calV(f)$ captures a logarithm of a derivative of $f$. 

Our main result (Theorem~\ref{thm:mainthm}) says, roughly speaking, that if the velocity field $u$ is bounded in $\dot{W}^{1,p}$ uniformly in time for some $p>1$, then $\calV\big(\th(t,\cdot)\big)$ grows at most linearly

\begin{equation}\label{eq:growth_of_v}
\calV\big(\th(t,\cdot)\big) \le C\, (1 + t)
\end{equation}

We also offer an additional, related result (Theorem~\ref{thm:W}). Consider the functional
\begin{equation*}
\calW(f)=\int_{\Rd} (\log\abs{\xi})^2\,\abss{\hat{f}(\xi)}^2\,d\xi
\end{equation*}
We show that if the velocity field $u$ is bounded in $\dot{W}^{1,q}$ uniformly in time for some $q\ge 2$, then $\calW\big(\th(t,\cdot)\big)$ grows at most quadratically
\begin{equation*}
\calW\big(\th(t,\cdot)\big) \le C(1+t)^2
\end{equation*}

Our results rely crucially on a harmonic analysis estimate recently proved by Seeger, Smart and Street~\cite{seeger}. It is clear from~\cite{seeger} that the results there are related to and motivated by Bressan's conjecture. However~\cite{seeger} does not include much information about these connections. 

Proving linear growth of $\calV\big(\th(t,\cdot)\big)$ recovers most of the existing results in the literature. We discuss this in detail in Section~\ref{sec:discussion}, giving just a brief summary here. A popular choice to measure mixing is the $\dot{H}^{-1}$ norm of $\th(t,\cdot)$ or more generally any $\dot{H}^{-s}$ norm for $s>0$ (see for instance~\cite{mathew2005multiscale} and the review article~\cite{thiffeault}). It was shown in~\cite{ikx} and~\cite{seis} that a uniform in time control of $\norm{\nabla u(t,\cdot)}_p$ implies an exponential lower bound on the $\dot{H}^{-1}$ norm of $\th(t,\cdot)$
\[
\normhmo{\th(t,\cdot)} \ge C^{-1} \exp(-Ct)
\]
This exponential decay can be recovered by our main theorem coupled to the simple convexity inequality (see Prop.~\ref{prop:exp_decay_functional} in Sect.~\ref{sec:recover_results})
\[
\normhms{f}/\normltwo{f} \ge \exp\big(\!-s\,\calV(f)/\normltwo{f}^2\big)
\]
Note that this gives us access to any $\dot{H}^{-s}$ norm for $s>0$.

To give a more precise statement of our main result we turn to a more precise description of the problem.

\section{Statement of main result}
Consider the functional $\calV$, defined on functions $f$ in the Schwartz class by
\begin{equation*}
\calV(f) = \int_{\Rd} \log{\abs{\xi}}\,\abss{\hat{f}(\xi)}^2\,d\xi
\end{equation*}
where $\hat{f}$ denotes the Fourier transform of $f$. This can be written in physical space (see Lemma~\ref{lemma:V_physical_space})
\begin{equation*}
\calV(f) = \alpha_d\left(\frac{1}{2}\iint_{\abs{x-y} \le 1} \frac{\abs{f(x)-f(y)}^2}{\abs{x-y}^d}\,dx\,dy - \iint_{\abs{x-y}>1} \frac{f(x)f(y)}{\abs{x-y}^d}\,dx\,dy\right) + \beta_d \normltwo{f}^2
\end{equation*}
where $\alpha_d,\beta_d$ are constants and $\alpha_d>0$.

Let $\th$ be a scalar quantity passively advected by a smooth divergence-free time-dependent velocity field $u$ (equation~\eqref{eq:conservation_of_mass}). A careful computation (see Lemma~\ref{lemma:time_derivative_V}) shows that the time-derivative of $\calV\big(\th(t,\cdot)\big)$ is
\begin{equation} \label{eq:time_derivative_V}
\frac{d}{dt}\calV\big(\th(t,\cdot)\big) = c_d\; \text{PV}\iint \th(t,x)\th(t,y)\big(u(t,x)-u(t,y)\big)\cdot\frac{x-y}{\abs{x-y}^{d+2}}\,dx\,dy
\end{equation}
where the principal value means taking the limit $\eps\to 0$ of the integral over the domain $\abs{x-y}>\eps$ in $\Rd\times\Rd$. 

The time-derivative of $\calV\big(\th(t,\cdot)\big)$ can be written as a trilinear form in $\th(t,\cdot)$, $\th(t,\cdot)$ and $\nabla u(t,\cdot)$. Using the harmonic analysis estimate from Seeger et al.~\cite{seeger}, we deduce that the right hand side of~\eqref{eq:time_derivative_V} can be bounded by
\begin{equation*}
C(d,p)\,\norm{\th(t,\cdot)}_{\infty}\,\norm{\th(t,\cdot)}_{p'}\,\norm{\nabla u(t,\cdot)}_p
\end{equation*}
with $p>1$ and $1/p+1/p'=1$. We now state our main theorem.

\begin{theorem}\label{thm:mainthm}
Let $1<p\le\infty$ and $p'$ the dual Hölder exponent $p/(p-1)$. There exists a constant $C>0$ depending only on $p$ and the dimension $d$ such that
\begin{enumerate}[label={\alph*)}]
\item $\abs{d/dt\,\calV\big(\th(t,\cdot)\big)}\le C \norm{\th_0}_{\infty} \norm{\th_0}_{p'} \norm{\nabla u(t,\cdot)}_p$
\end{enumerate}
where $\norm{\th_0}_{\infty}=\normlinf{\th_0}$, etc. As consequences:
\begin{enumerate}[label={\alph*)},resume]
\item If $\th_0\in L^1\cap L^{\infty}$ and $u$ is bounded in $\dot{W}^{1,p}$ uniformly in time, then $\calV\big(\th(t,\cdot)\big)$ grows at most linearly;
\item More generally, $\calV\big(\th(t,\cdot)\big) - \calV\big(\th_0\big) \le C\norm{\th_0}_{\infty}\norm{\th_0}_{p'}\displaystyle\int_0^t \norm{\nabla u(t',\cdot)}_p\,dt'$
\end{enumerate}
\end{theorem}

We consider now another functional 
\begin{equation*}
\calW(f) = \int_{\Rd} (\log{\abs{\xi}})^2\,\abss{\hat{f}(\xi)}^2\,d\xi
\end{equation*}
which provides slightly stronger control of the high frequencies. We have the following bounds for $\calW$:
\begin{theorem} \label{thm:W}
Let $2\le q\le \infty$ and $2\le\tilde{q}\le\infty$ be such that $1/q+1/\tilde{q} = 1/2$. There exists a constant $C>0$ depending only on the dimension $d$ such that
\begin{enumerate}[label={\alph*)}]
\item $\abs{d/dt\sqrt{\calW\big(\th(t,\cdot)\big)}} \le C \norm{\th_0}_{\tilde{q}} \norm{\nabla u(t,\cdot)}_q$
\end{enumerate}
where $\norm{\th_0}_{\tilde{q}}=\norm{\th_0}_{L^{\tilde{q}}}$, etc. As a consequence:
\begin{enumerate}[label={\alph*)},resume]
\item If $\th_0\in L^2\cap L^{\infty}$ and $u$ is bounded in $\dot{W}^{1,q}$ uniformly in time, then $\calW\big(\th(t,\cdot)\big)$ grows at most quadratically.
\end{enumerate}
\end{theorem}

\begin{remark}
In Theorem~\ref{thm:W}, the best we can do is obtaining a bound in terms of the $L^2$ norm in space $\norm{\nabla u(t,\cdot)}_2$, whereas in Theorem~\ref{thm:mainthm} we can go all the way down to $\norm{\nabla u(t,\cdot)}_p$ for $p>1$. 
\end{remark}

Further discussion of our results and some corollaries are given in Section~\ref{sec:discussion}.

\section{Proofs}

In this section we prove Theorems~\ref{thm:mainthm} and~\ref{thm:W}. We start by expressing $\calV(f)$ in physical space.

\begin{lemma} \label{lemma:V_physical_space}
For any $f$ in the Schwartz class, $\calV(f)$ can be written in physical space
\begin{equation*}
\calV(f) = \alpha_d\left(\frac{1}{2}\iint_{\abs{x-y} \le 1} \frac{\abs{f(x)-f(y)}^2}{\abs{x-y}^d}\,dx\,dy - \iint_{\abs{x-y}>1} \frac{f(x)f(y)}{\abs{x-y}^d}\,dx\,dy\right) + \beta_d \normltwo{f}^2
\end{equation*}
where $\alpha_d$ and $\beta_d$ are two constants and $\alpha_d >0$. 
\end{lemma}

\begin{proof}
Define the following tempered distribution $T$
\[
\bracket{T,\phi} = \int_{\abs{h}\le 1}\frac{\phi(h)-\phi(0)}{\abs{h}^d}\,dh + \int_{\abs{h}>1}\frac{\phi(h)}{\abs{h}^d}\,dh
\]
Then (see the appendix) the Fourier transform of $T$ is an $L^1_{\text{loc}}$ function and 
\[
\bracket{\hat{T},\psi} = \int\big(\zeta_d-\sigma_{d-1}\log\abs{\xi}\big)\,\psi(\xi)\,d\xi
\]
for any $\psi$ in the Schwartz class $\Schwartz(\Rd)$, where $\sigma_{d-1}$ is the surface area of the unit sphere in $\Rd$ and $\zeta_d$ is a constant. Consequently if we define the tempered distribution $S$ by
\[
S = \zeta_d/\sigma_{d-1}\delta_0 - 1/\sigma_{d-1} T
\]
then for all $\psi\in\Schwartz(\Rd)$ we have
\[
\bracket{\hat{S},\psi}=\int\log\abs{\xi}\psi(\xi)\,d\xi
\]
We deduce that for all $f\in\Schwartz(\Rd)$, writing $\tilde{f}(x)=f(-x)$ we have
\begin{align*}
\calV(f) &= \bracket{\hat{S},\hat{f}\bar{\hat{f}}} = \bracket{\hat{S},\widehat{f\conv\tilde{f}}} = \bracket{\hat{\hat{S}},f\conv\tilde{f}}\\
 &= \zeta_d/\sigma_{d-1}(f\conv\tilde{f})(0) - 1/\sigma_{d-1}\begin{aligned}[t] & \left(\int_{\abs{h}\le 1}\int\frac{\big(f(x-h)-f(x)\big)f(x)}{\abs{h}^d}\,dx\,dh\right.\\
 & \left. + \int_{\abs{h} > 1}\int\frac{f(x-h)f(x)}{\abs{h}^d}\,dx\,dh \right)\end{aligned}\\
 &= \zeta_d/\sigma_{d-1}\normltwo{f}^2 - 1/\sigma_{d-1}\begin{aligned}[t] & \left(-\frac{1}{2}\iint_{\abs{x-y}\le 1}\frac{\abs{f(x)-f(y)}^2}{\abs{x-y}^d}\,dx\,dy\right.\\
 & \left. + \iint_{\abs{x-y}>1}\frac{f(y)f(x)}{\abs{x-y}^d}\,dx\,dy\right)\end{aligned}
\end{align*}
which concludes the proof.
\end{proof}

We now give the expression of the time-derivative of $\calV\big(\th(t,\cdot)\big)$.
\begin{lemma}\label{lemma:time_derivative_V}
For smooth divergence-free velocity fields $u$ decaying fast enough at infinity, the time-derivative of $\calV\big(\th(t,\cdot)\big)$ can be written
\begin{equation*}
\frac{d}{dt}\calV\big(\th(t,\cdot)\big) = c_d\;\text{PV}\iint \th(t,x)\th(t,y)\big(u(t,x)-u(t,y)\big)\cdot\frac{x-y}{\abs{x-y}^{d+2}}\,dx\,dy
\end{equation*}
where $c_d$ is a positive constant. 
\end{lemma}

\begin{proof}
The derivation has to be done a bit carefully otherwise non integrable terms appear. Recall that the expression of $\calV$ in physical space is (Lemma~\ref{lemma:V_physical_space})
\begin{multline*}
\calV\big(\th(t,\cdot)\big) = \alpha_d\left(\frac{1}{2}\iint_{\abs{x-y}\le 1} \frac{\abs{\th(t,x)-\th(t,y)}^2}{\abs{x-y}^d}\,dx\,dy\right.\\ - \left.\iint_{\abs{x-y}>1} \frac{\th(t,x)\th(t,y)}{\abs{x-y}^d}\,dx\,dy\right) + \beta_d\normltwo{\th(t,\cdot)}^2
\end{multline*}
Taking time-derivatives, the $L^2$ norm term on the right-hand side disappears since the flow is incompressible. Dropping the $t$ for readability we get
\begin{align*}
\frac{d}{dt}\calV(\th) &= \begin{aligned}[t]\alpha_d\;\left(\iint_{\abs{x-y}\le 1} \frac{\big(\th(x)-\th(y)\big)\big(\partial_t\th(x)-\partial_t\th(y)\big)}{\abs{x-y}^d}\,dx\,dy\right.\\- \left.\iint_{\abs{x-y}>1}\frac{\partial_t\th(x)\th(y) + \th(x)\partial_t\th(y)}{\abs{x-y}^d}\,dx\,dy\right)\end{aligned}\\
 &= \alpha_d \; \big(A - B\big)
\end{align*}
Note that since $\th$ and $\partial_t\th$ are smooth and decay fast at infinity, there is no issue in differentiating under the integral sign.

Set 
\[
V(h) = \frac{1}{\abs{h}^d}
\]
We first simplify $A$. By symmetry in the integral in $x$ and $y$ we only keep the term with $\partial_t\th(x)$ and write
\begin{equation*}
\frac{A}{2} = \iint_{\abs{x-y}\le 1} V(x-y)\big(\th(x)-\th(y)\big)\big(-\div(u\th)(x)\big)\,dx\,dy
\end{equation*}
Next we want to integrate by part but doing so directly yields a term which is not integrable. Thus we proceed in the following way:
\begin{align*}
\frac{A}{2} &= \lim_{\eps\to 0} \int_y\int_{\eps < \abs{x-y} \le 1} V(x-y)\big(\th(x)-\th(y)\big)\big(-\div(u\th)(x)\big)\,dx\,dy\\
 &= \begin{aligned}[t] &\lim_{\eps\to 0} \left(\int_y\int_{\eps <\abs{x-y} \le 1} \Big(\nabla V(x-y)\big(\th(x)-\th(y)\big) + V(x-y)\nabla\th(x)\Big)\cdot u(x)\th(x)\,dx\,dy\right.\\&- \int_y\int_{x\in S(y,1)} V(x-y)\big(\th(x)-\th(y)\big)\th(x)\,u(x)\cdot\frac{x-y}{\abs{x-y}}\,d\sigma(x)\,dy\\ &\left.- \int_y\int_{x\in S(y,\eps)} V(x-y)\big(\th(x)-\th(y)\big)\th(x)\,u(x)\cdot\frac{y-x}{\abs{y-x}}\,d\sigma(x)\,dy\right)\end{aligned}\\
 &= \lim_{\eps\to 0} \big(A_1(\eps) + A_2 + A_3(\eps)\big)
\end{align*}
where $S(y,1)$ denotes the sphere of center $y$ and of radius $1$, etc. 

We now fix $\eps>0$ and compute each of the terms $A_1(\eps)$, $A_2$, $A_3(\eps)$. The first one
\[
A_1(\eps)=\int_y\int_{\eps <\abs{x-y} \le 1} \Big(\nabla V(x-y)\big(\th(x)-\th(y)\big) + V(x-y)\nabla\th(x)\Big)\cdot u(x)\th(x)\,dx\,dy
\]
can be split into three terms (we swap integrals in $x$ and $y$)
\begin{equation*}
A_1(\eps) = \begin{aligned}[t]&\int_x \th^2(x)\, u(x)\cdot\int_{\eps<\abs{y-x}\le 1} \nabla V(x-y)\,dy\,dx\\
		  &- \iint_{\eps<\abs{x-y}\le 1} \nabla V(x-y)\cdot u(x)\,\th(x)\,\th(y)\,dx\,dy\\
		  &+ \int_x\frac{1}{2}\nabla\big(\th^2\big)(x)\cdot u(x)\int_{\eps<\abs{y-x}\le 1} V(x-y)\,dy\,dx
		  \end{aligned}
\end{equation*}
The first term cancels since the integral in $y$ is zero, and the last term cancels since the integral in $y$ doesn't depend on $x$ and $u$ is divergence-free. Thus we are left with (after symmetrizing in $x$ and $y$ in the second term)
\[
A_1(\eps) = -\frac{1}{2}\iint_{\eps<\abs{x-y}\le 1} \th(x)\th(y)\big(u(x)-u(y)\big)\cdot\nabla V(x-y)\,dx\,dy
\]

We now turn our attention to $A_2$ (which doesn't depend on $\eps$)
\[
A_2 = - \int_y\int_{x\in S(y,1)} V(x-y)\big(\th(x)-\th(y)\big)\th(x)\,u(x)\cdot\frac{x-y}{\abs{x-y}}\,d\sigma(x)\,dy
\]
We split the integral in two and swap the integrals in $x$ and $y$ in the first term
\begin{equation*}
A_2 = \begin{aligned}[t] &- \int_x \th^2(x)\,u(x) \cdot \int_{y\in S(x,1)}V(x-y)\frac{x-y}{\abs{x-y}}\,d\sigma(y)\,dx\\
	  &+ \int_y\int_{x\in S(y,1)} V(x-y)\th(y)\th(x)\,u(x)\cdot\frac{x-y}{\abs{x-y}}\,d\sigma(x)\,dy
	  \end{aligned}
\end{equation*}
The first term cancels since the integral in $y$ is zero. Thus we are left with
\[
A_2 =  \int_y\int_{x\in S(y,1)} V(x-y)\th(y)\th(x)\,u(x)\cdot\frac{x-y}{\abs{x-y}}\,d\sigma(x)\,dy
\]

Now we simplify the third term
\[
A_3(\eps) = \int_y\int_{x\in S(y,\eps)} V(x-y)\big(\th(y)-\th(x)\big)\th(x)\,u(x)\cdot\frac{y-x}{\abs{y-x}}\,d\sigma(x)\,dy
\]
and show that
\[
\lim_{\eps\to 0} A_3(\eps) = 0
\]
We use the following Taylor's formula with integral remainder to estimate the quantity $\th(y)-\th(x)$
\[
\th(y)-\th(x)=\nabla\th(x)\cdot(y-x) + \int_0^1(1-s)D^2\th\big(x+s(y-x)\big)(y-x,y-x)\,ds
\]
and plug it in the expression of $A_3(\eps)$. $A_3(\eps)$ decomposes into two terms
\[
A_3'(\eps) = \int_y\int_{x\in S(y,\eps)} V(x-y)\nabla\th(x)\cdot(y-x)\,\th(x)\,u(x)\cdot\frac{y-x}{\abs{y-x}}\,d\sigma(x)\,dy
\]
and
\[
A_3''(\eps) = \int_y\int_{x\in S(y,\eps)} V(x-y) \int_0^1(1-s)D^2\th\big(x+s(y-x)\big)(y-x,y-x)ds\; \th(x)\,u(x)\cdot\frac{y-x}{\abs{y-x}}\,d\sigma(x)\,dy
\]
Swapping the integrals in $x$ and $y$ and doing a change of variables $h=y-x$ yields
\[
A_3'(\eps) = \int_x\int_{h\in S(0,\eps)} V(h)\nabla\big(\th^2(x)/2\big)\cdot h\; u(x)\cdot\frac{h}{\abs{h}}\,d\sigma(h)\,dx
\]
Rearranging terms, we get
\[
A_3'(\eps) = \int_xu_i(x)\partial_j\big(\th^2/2\big)(x) \int_{h\in S(0,\eps)} V(h) h_j\frac{h_i}{\abs{h}}\,d\sigma(h)\,dx
\]
where the summation over indices $i$ and $j$ is implied. In the integral in $h$, note that $\abs{h}=\eps$ and $V(h)=\eps^{-d}$. Furthermore, by rotationnal symmetry of the sphere $S(0,\eps)$ it is easy to see that
\[
\int_{h\in S(0,\eps)}h_i\,h_j\,d\sigma(h) = \begin{cases}
0\quad\text{if } i\neq j\\
d^{-1}\sigma_{d-1}\eps^{d+1}\quad\text{if }i=j
\end{cases}
\]
Consequently
\begin{align*}
A_3'(\eps) &= d^{-1}\sigma_{d-1}\int u_i(x)\partial_i\big(\th^2/2\big)(x)\,dx\\
 &= 0
\end{align*}
since $u$ is divergence-free.

Finally, we can crudely bound the term $A_3''(\eps)$
\begin{align*}
\abs{A_3''(\eps)} &= \begin{aligned}[t]\left\lvert\int_x\int_{y\in S(x,\eps)} V(x-y) \int_0^1(1-s)D^2\th\big(x+s(y-x)\big)(y-x,y-x)\,ds\right.\\
\left.\th(x)\,u(x)\cdot\frac{y-x}{\abs{y-x}}\,d\sigma(y)\,dx\right\rvert\end{aligned}\\
 &\le \int_x\int_{y\in S(x,\eps)} V(x-y) \normlinf{D^2\th}\abs{y-x}^2 \abs{\th(x)}\abs{u(x)}\,d\sigma(y)\,dx\\
 &\le \int_x \sigma_{d-1}\eps^{d-1} \eps^{-d} \normlinf{D^2\th}\eps^2 \abs{\th(x)}\abs{u(x)}\,dx\\
 &\le \eps\,\sigma_{d-1}\normlone{u}\normlinf{\th}\normlinf{D^2\th}
\end{align*}
This shows that $A_3''(\eps)\to 0$ as $\eps\to 0$ and thus
\[
A_3(\eps)=A_3'(\eps)+A_3''(\eps) \to 0\quad\text{as }\eps\to 0
\]

We now deal with the $B$ term. By symmetry in $x$ and $y$ we have
\[
\frac{B}{2} = \iint_{\abs{x-y}>1} V(x-y)\partial_t\th(x)\th(y)\,dx\,dy
\]
Here we can substitute $\partial_t\th$ by $-\div(u\th)$ and directly do an integration by part
\begin{align*}
\frac{B}{2} &= \int_y\th(y)\int_{\abs{x-y}>1} V(x-y) \big(-\div(u\th)\big)(x)\,dx\,dy\\
 &= \begin{aligned}[t] & \iint_{\abs{x-y}>1} \th(y)\nabla V(x-y)\cdot u(x)\th(x)\,dx\,dy\\
    & - \int_y\th(y)\int_{x\in S(y,1)} V(x-y)\th(x)u(x)\cdot\frac{y-x}{\abs{y-x}}\,d\sigma(x)\,dy
	\end{aligned}\\
 &= B_1 + B_2
\end{align*}

Note that $B_2=A_2$. Moreover, after symmetrizing in $x$ and $y$, $B_1$ can be written
\[
B_1 =\frac{1}{2} \iint_{\abs{x-y}>1} \th(x)\th(y)\big(u(x)-u(y)\big)\cdot\nabla V(x-y)\,dx\,dy  
\]

In conclusion, grouping everything together we get
\begin{align*}
\frac{A-B}{2} &= \lim_{\eps\to 0}\big(A_1(\eps) + A_2 + A_3(\eps)\big) - (B_1 + B_2)\\
 &= \lim_{\eps\to 0} A_1(\eps) - B_1\\
 &= \lim_{\eps\to 0} -\frac{1}{2} \iint_{\eps<\abs{x-y}\le 1} \th(x)\th(y)\big(u(x)-u(y)\big)\cdot\nabla V(x-y)\,dx\,dy \\
 &- \frac{1}{2} \iint_{\abs{x-y}> 1} \th(x)\th(y)\big(u(x)-u(y)\big)\cdot\nabla V(x-y)\,dx\,dy
\end{align*}

Hence
\[
A-B = \lim_{\eps\to 0} -\iint_{\abs{x-y}>\eps} \th(x)\th(y)\big(u(x)-u(y)\big)\cdot\nabla V(x-y)\,dx\,dy 
\]
which concludes the proof since $\nabla V(h) = -d \; h/\abs{h}^{d+2}$.
\end{proof}

We are now able to prove Theorem~\ref{thm:mainthm}.

\begin{proof}[Proof of Theorem~\ref{thm:mainthm}]
With some notation we can write the time derivative
\begin{equation*}
\frac{d}{dt}\calV\big(\th(t,\cdot)\big) = c_d\;\text{PV}\iint \th(t,x)\th(t,y)\Big(u(t,x)-u(t,y)\Big)\cdot\frac{x-y}{\abs{x-y}^{d+2}}\,dx\,dy
\end{equation*}
as the type of multilinear singular integral studied in~\cite{seeger}. Define
\begin{equation*}
K(h) = c_d\;\frac{h\otimes h-1/d\abs{h}^2 I}{\abs{h}^{d+2}}
\end{equation*}
for $h\in\Rd\setminus\{0\}$, where $h\otimes h$ is the $d\times d$ matrix defined by 
\[
(h\otimes h)_{i,j} = h_ih_j
\]
and $I$ is the $d\times d$ identity matrix. 

It is easy to see that $K$ is a matrix-valued Calderón-Zygmund kernel (i.e. each entry $K_{i,j}$ is a Calderón-Zygmund kernel). Let $m_{x,y}L$ denote the average of a function $L$ between $x$ and $y$ 
\[
m_{x,y}L = \int_0^1 L\big((1-s)x+s y\big)\,ds
\]
and let $M:N$ denote the contraction $M:N=\sum_{i,j}M_{i,j}N_{i,j}$.

Thanks to the divergence-free condition  $\nabla u : I = 0$, where $\big(\nabla u\big)_{i,j} = \partial_iu_j$, we can write (summing over repeated indices)
\begin{align*}
\big(u_j(t,x)-u_j(t,y)\big)\frac{x_j-y_j}{\abs{x-y}^{d+2}} &= \int_0^1 \nabla u_j\big(t,(1-s)x+sy\big)\cdot(x-y)\,ds\frac{x_j-y_j}{\abs{x-y}^{d+2}}\\
 &= m_{x,y}\big(\nabla u(t,\cdot)\big) : \frac{(x-y)\otimes(x-y)-1/d\abs{x-y}^2 I}{\abs{x-y}^{d+2}}
\end{align*}
and thus
\begin{equation*}
\frac{d}{dt}\calV\big(\th(t,\cdot)\big) = \iint \th(t,x)\th(t,y)\,m_{x,y}\big(\nabla u(t,\cdot)\big):K(x-y)\,dx\,dy
\end{equation*}

Using the terminology from~\cite{seeger}, this is a first order $d$-commutator. The main result of~\cite{seeger} is that Hölder estimates are valid on this types of trilinear form. More precisely we have the bound
\[
 \iint \th(t,x)\th(t,y)\,m_{x,y}\big(\nabla u(t,\cdot)\big):K(x-y)\,dx\,dy \le C(d,p) \normlinf{\th(t,\cdot)}\,\,\norm{\th(t,\cdot)}_{L^{p'}}\,\normlp{\nabla u(t,\cdot)}
\]
for any $p>1$.

The $L^{p'}$ and $L^{\infty}$ norms of $\th(t,\cdot)$ are conserved quantities since the flow is incompressible, and this concludes the proof of Theorem~\ref{thm:mainthm}.
\end{proof}

To prove Theorem~\ref{thm:W} we start with the following lemma

\begin{lemma}[Time-derivative of $\calW\big(\th(t,\cdot)\big)$] \label{lemma:time_derivative_W}
For smooth divergence-free velocity fields $u$ decaying fast enough at infinity, the time-derivative of $\calW\big(\th(t,\cdot)\big)$ can be written
\begin{equation*}
\frac{d}{dt}\calW\big(\th(t,\cdot)\big) = c_d\;\text{PV}\iint \phi(t,x)\th(t,y)\big(u(t,x)-u(t,y)\big)\cdot\frac{x-y}{\abs{x-y}^{d+2}}\,dx\,dy
\end{equation*}
where we denote
\[
\hat{\phi}(t,\xi) = \log(\abs{\xi})\,\hat{\th}(t,\xi)
\]
and $c_d$ is the same constant as in Lemma~\ref{lemma:time_derivative_V}.
\end{lemma}

\begin{proof}
Let $f=f(x)$ be any functions (regular enough, e.g. in the Schwartz class) and $u=u(t,x)$ be any divergence-free velocity field (also regular enough). Let $\th$ be the solution of
\begin{align*}
\begin{cases}
\partial_t\th+\div(u\th) = 0 &\text{on } (-1,1)\times\Rd\\
\th(0,\cdot) = f & \text{on } \Rd
\end{cases}
\end{align*}
and denote
\[
v(x) = u(0,x)
\]
Then Lemma~\ref{lemma:time_derivative_V} shows that
\begin{equation} \label{eq:trilinear_form}
\left.\frac{d}{dt}\right\rvert_{t=0} \calV\big(\th(t,\cdot)\big) =  c_d\;\text{PV}\iint f(x)f(y)\big(v(x)-v(y)\big)\cdot\frac{x-y}{\abs{x-y}^{d+2}}\,dx\,dy
\end{equation}
On the other hand, computing this time-derivative in Fourier space gives us
\begin{align*}
\frac{d}{dt}\calV\big(\th(t,\cdot)\big) &= \frac{d}{dt}\int \log\abs{\xi}\,\abss{\hat{\th}(t,\xi)}^2\,d\xi\\
 &= 2\Re\int\log\abs{\xi} \partial_t\hat{\th}(t,\xi)\overline{\hat{\th}(t,\xi)}\,d\xi\\
 &= 2\Re\int\log\abs{\xi} (-i)\xi\cdot(\hat{u}\conv\hat{\th}(t,\cdot))\hat{\th}(t,-\xi)\,d\xi\\
 &= 2\Re\iint\log\abs{\xi} (-i)\xi\cdot\hat{u}(t,\xi+\eta)\hat{\th}(t,-\eta)\hat{\th}(t,-\xi)\,d\xi\,d\eta
\end{align*}
where $\Re$ denotes the real part. Writing this last expression at time $t=0$ and using~\eqref{eq:trilinear_form}, we get an equivalence formula in physical and Fourier space:
\[
c_d\;\text{PV}\iint f(x)f(y)\big(v(x)-v(y)\big)\cdot\frac{x-y}{\abs{x-y}^{d+2}}\,dx\,dy =  2\Re\iint\log\abs{\xi} (-i)\xi\cdot\hat{v}(\xi+\eta)\hat{f}(-\xi)\hat{f}(-\eta)\,d\xi\,d\eta
\]
true for any $f$ and divergence-free $v$; thus if we polarize the following equality is true
\begin{multline} \label{eq:identity_physical_Fourier}
c_d\;\text{PV}\iint f(x)\,g(y)\,\big(v(x)-v(y)\big)\cdot\frac{x-y}{\abs{x-y}^{d+2}}\,dx\,dy =\\  2\Re\iint\log\abs{\xi} (-i)\,\xi\cdot\hat{v}(\xi+\eta)\,\hat{f}(-\xi)\,\hat{g}(-\eta)\,d\xi\,d\eta
\end{multline}
for any (smooth, fast-decaying) $f$, $g$ and divergence-free $v$.

Furthermore, in Fourier space  the time-derivative of $\calW\big(\th(t,\cdot)\big)$ is (computations are similar to the previous ones)
\begin{align*}
\frac{d}{dt}\calW\big(\th(t,\cdot)\big) &= 2\Re\iint\big(\log\abs{\xi}\big)^2 (-i)\xi\cdot\hat{u}(t,\xi+\eta)\hat{\th}(t,-\eta)\hat{\th}(t,-\xi)\,d\xi\,d\eta\\
 &=  2\Re\iint\log\abs{\xi} (-i)\xi\cdot\hat{u}(t,\xi+\eta)\hat{\th}(t,-\eta)\hat{\phi}(t,-\xi)\,d\xi\,d\eta
\end{align*}
where
\[
\hat{\phi}(t,\xi) = \log(\abs{\xi})\,\hat{\th}(t,\xi)
\]
Using equality~\eqref{eq:identity_physical_Fourier} then yields the desired result.
\end{proof}

\begin{proof}[Proof of Theorem~\ref{thm:W}]
Lemma~\ref{lemma:time_derivative_W} just showed that
\begin{equation*}
\frac{d}{dt}\calW\big(\th(t,\cdot)\big) = c_d\;\text{PV}\iint \phi(t,x)\th(t,y)\Big(u(t,x)-u(t,y)\Big)\cdot\frac{x-y}{\abs{x-y}^{d+2}}\,dx\,dy
\end{equation*}
where $\hat{\phi}(t,\xi) = \log(\abs{\xi})\,\hat{\th}(t,\xi)$. This is the same first order $d$-commutator considered previously (see the proof of Theorem~\ref{thm:mainthm}), but this time in $\phi(t,\cdot)$, $\th(t,\cdot)$ and $\nabla u(t,\cdot)$. The results in~\cite{seeger} imply
\begin{multline*}
 \iint \phi(t,x)\th(t,y)\Big(u(t,x)-u(t,y)\Big)\cdot\frac{x-y}{\abs{x-y}^{d+2}}\,dx\,dy\le \\
 C(d) \normltwo{\phi(t,\cdot)}\,\,\norm{\th(t,\cdot)}_{L^{\tilde{q}}}\,\norm{\nabla u(t,\cdot)}_{L^q}
\end{multline*}
for any $2\le q\le \infty$, where $2\le \tilde{q}\le \infty$ is such that $1/2+1/q+1/\tilde{q} = 1$. Note that the constant $C(d)$ doesn't depend on $q$. It is immediate to see in Fourier space that
\[
\normltwo{\phi(t,\cdot)} = \sqrt{\calW\big(\th(t,\cdot)\big)}
\]
which is enough to conclude to proof.
\end{proof}

\section{Discussion and corollaries} \label{sec:discussion}

\subsection{The functionals $\calV$ and $\calW$} \label{sec:functionals}
Let us give some background on the functionals $\calV$ and $\calW$ introduced in this work. A regularized version of $\calV$ was considered in~\cite{jabin}, on a problem that shares similarities with the one considered in this paper. In~\cite{crippadelellis}, the idea of controlling the $\log$ of a derivative of $\th(t,\cdot)$ is also present. Additionally, the usual homogeneous Sobolev seminorm $\norm{f}_{\dot{H}^s}$ can be defined in Fourier space by
\[
\norm{f}_{\dot{H}^s}^2=\int_{\Rd}\abs{\xi}^{2s}\,\abss{\hat{f}(\xi)}^2\,d\xi
\]
This can be written in physical space for $0<s<1$
\[
\norm{f}_{\dot{H}^s}^2=c_{d,s} \iint_{\Rd\times\Rd} \frac{\abs{f(x)-f(y)}^2}{\abs{x-y}^{d+2s}}\,dx\,dy
\]
where the constant $c_{d,s}$ blows up as $s\to 0$ (see for instance~\cite{mazyashaposhnikova} for more information). The functionals $\calV$ and $\calW$ involve taking a limit $s\to 0$. More precisely, note that formally for small $s$
\[
\abs{\xi}^{2s} = 1 + 2s\log\abs{\xi} + 2s^2(\log{\abs{\xi}})^2+O(s^3)
\]
Thus $\calV(f)$ and $\calW(f)$ are the first terms of the expansion of $\norm{f}_{\dot{H}^s}^2$ as $s\to 0$
\[
\norm{f}_{\dot{H}^s}^2 = \normltwo{f}^2 + 2s\,\calV(f) + 2s^2\,\calW(f) + O(s^3)
\]

See Section~\ref{sect:blowup_sobolev} for additional insight.

\subsection{Brief comments on the harmonic analysis estimate}
For our theorems~\ref{thm:mainthm} and~\ref{thm:W} we rely crucially on a hard harmonic analysis estimate; more precisely we need Hölder-type bounds on a bilinear singular integral. Without going in too much depth, let us point out that usual Calderón-Zygmund theory does not apply here as the kernel is too singular (if $d\ge 2$) and even the multilinear singular integral framework developed by Grafakos and Torres~\cite{grafakostorres} is not adapted for this problem. 

In~\cite{christjourne} Christ and Journé studied the type of multilinear singular integral operators we consider in this paper, however applying their results would only give us a bound on $d/dt\,\calV\big(\th(t,\cdot)\big)$ in terms of the $L^{\infty}$ norm of $\nabla u(t,\cdot)$. This is not useful here as controlling $\norm{\nabla u(t,\cdot)}_{\infty}$ makes everything obvious; thus only the very recent work~\cite{seeger} contains the needed estimates. We refer to~\cite{seeger} for more information on these types of multilinear singular integral operators.

\subsection{Two ways to quantify mixing} \label{sec:mixing_def}
Two measures of mixing have been mainly considered in the literature. We follow here the naming of~\cite{acm}.

\begin{definition}
The functional mixing scale of $\th(t,\cdot)$ is $\normhmo{\th(t,\cdot)}$.
\end{definition}

The second notion has traditionally been considered for functions with values $\pm 1$, but it can be extended naturally to functions in $L^{\infty}$

\begin{definition} \label{def:geometric_mixing_scale}
Given $0<\kappa<1$, the geometric mixing scale of $\th(t,\cdot)$ is the infimum $\eps(t)$ of all $\eps>0$ such that
\begin{equation*}
\frac{\normlinf{\th(t,\cdot)\conv \chi_{\eps}}}{\normlinf{\th(t,\cdot)}} \le 1-\kappa
\end{equation*}
where $\conv$ denotes the convolution, $\chi$ is the indicator function of the unit ball in $\Rd$ modified to have total mass $1$: $\chi(x)=\frac{1}{\abs{B_1(0)}}\ind_{B_1(0)}(x)$ and
\[
\chi_{\eps}(x) = \eps^{-d}\chi(x/\eps)
\]
\end{definition}

\subsection{Relations between our results and previous ones} \label{sec:recover_results}
A consequence of the main result in~\cite{crippadelellis} is an exponential lower bound on the geometric mixing scale 
\[
\eps(t) \ge C^{-1}\exp\left(-C\int_0^t\lVert\nabla u(t',\cdot)\,dt'\rVert_{L^p}\right)
\]
for $p>1$, while the works~\cite{ikx},\cite{seis} showed an exponential lower bound on the functional mixing scale 
\[
\normhmo{\th(t,\cdot)} \ge C^{-1}\exp\left(-C\int_0^t\lVert\nabla u(t',\cdot)\,dt'\rVert_{L^p}\right)
\]

We recover both results. Indeed upper bounds on $\calV(f)$ imply lower bounds on the functional and geometric mixing scales. More precisely, for the functional mixing scale we have the following proposition

\begin{prop}[Exponential decay of the functional mixing scale] \label{prop:exp_decay_functional}
For all $s>0$ we have the following convexity inequality
\begin{enumerate}[label={\alph*)}]
\item For all non zero $f$ in the Schwarz class $\Schwartz$
\begin{equation*}
\normhms{f}/\normltwo{f} \ge \exp\big(-s\,\,\calV(f)/\normltwo{f}^2\big)
\end{equation*}
\end{enumerate}
As a consequence:
\begin{enumerate}[label={\alph*)},resume]
\item Fix any $s>0$. There is a constant $C$ depending only on $p$ and the dimension $d$ such that
\begin{equation*}
\normhms{\th(t,\cdot)} \ge \normltwo{\th_0}\!\exp\big(\!-s\calV(\th_0)/\!\normltwo{\th_0}\!\big)\,\exp\left( - C\frac{\normlinf{\th_0}\norm{\th_0}_{L^{p'}}}{\normltwo{\th_0}^2}\int_0^t\normlp{\nabla u(t',\cdot)}\,dt'\right)
\end{equation*}
where $p'=p/(p-1)$, for any smooth,fast-decaying solution $\th$ of~\eqref{eq:conservation_of_mass}.
\end{enumerate}
\end{prop}

\begin{remark}
The decay rate we obtain this way is
\[
C(d,p) \, \frac{\normlinf{\th_0}\norm{\th_0}_{L^{p'}}}{\normltwo{\th_0}^2}
\]
For $p=2$ for instance we get decay rate 
\[
C(d,p) \frac{\normlinf{\th_0}}{\normltwo{\th_0}}
\]
which is a slight improvement from~\cite{ikx}, where the authors find the following decay rate (in the $L^2$ case $p=2$)
\[
\frac{c(d,p)}{\abs{A_{\lambda}}^{1/2}}
\]
where $A_{\lambda}$ is the set $\left\{x\mid\th_0(x)/\normlinf{\th_0} > \lambda\right\}$ and $\abs{A_{\lambda}}$ denotes its Lebesgue measure. The parameter $\lambda$ is a fixed number in $(0,1)$. Note that a weak $L^2$ estimate yields 
\[
\abs{A_{\lambda}}^{1/2} \le \lambda^{-1} \frac{\normltwo{\th_0}}{\normlinf{\th_0}}
\]
so that
\[
\frac{c(d,p)}{\abs{A_{\lambda}}^{1/2}} \ge c(d,p) \lambda \frac{\normlinf{\th_0}}{\normltwo{\th_0}}
\]

In any case, both decay rates are larger when the support of $\th_0$ is smaller, which matches the intuition, see the related discussion in~\cite[Sect.~1]{ikx}.
\end{remark}

For the geometric mixing scale we have the following proposition
\begin{prop}[Exponential decay of the geometric mixing scale] \label{prop:exp_decay_geometric}
~
\begin{enumerate}[label={\alph*)}]
\item There exists a constant $A>0$ depending on the dimension $d$ and on $\kappa$ (see def.~\ref{def:geometric_mixing_scale}) such that for any $f\in\Schwartz$ with $\calV(f)>0$
\begin{equation*}
\eps<A^{-1}\,\exp\big(-A\,\calV(f)/\normltwo{f}^2\big) \imply \frac{\normlinf{f\conv\chieps}}{\normlinf{f}} > (1-\kappa) \frac{\normltwo{f}^2}{\normlone{f}\normlinf{f}}
\end{equation*}
\end{enumerate}
As a consequence:
\begin{enumerate}[label={\alph*)},resume]
\item Consider the system~\eqref{eq:conservation_of_mass} on the flat torus $\Tor^d$. Assume that $\th_0$ only has values $\pm 1$. Then there exist constants $A(d,\kappa)$ and $C(d,p)$ such that the geometric mixing scale $\eps(t)$ satisfies 
\begin{equation*}
\eps(t)\ge A^{-1} \exp(-A\calV(\th_0))\exp\left(-A\,C\int_0^t\normlp{\nabla u(t',\cdot)}\,dt'\right)
\end{equation*}
\end{enumerate}
\end{prop}

\begin{remark}\label{rem:torus}
For part~\emph{b)} of Prop.~\ref{prop:exp_decay_geometric} we consider the dynamic on the torus $\Tor^d$. To be rigorous we need to define a version $\widetilde{\calV}$ of the functional $\calV$ which acts on (smooth) functions defined on the torus
\[
\widetilde{\calV}(f) = \sum_{k\in\Integ^d\setminus\{0\}} \log\abs{k}\,\abss{\hat{f}(k)}^2
\]
where
\[
\hat{f}(k) = \int_{\Tor^d}e^{-2i\pi k\cdot x}f(x)\,dx
\]
We have the same bounds on $\widetilde{\calV}$ as for $\calV$ in Theorem~\ref{thm:mainthm}, since we can compute the time-derivative similarly to Lemma~\ref{lemma:time_derivative_V} (which only uses integrations by part) and the harmonic analysis estimate from~\cite{seeger} is also valid for $\widetilde{\calV}$.
\end{remark}

We now prove both propositions.

\begin{proof}[Proof of Prop.~\ref{prop:exp_decay_functional}]
Let us consider the probability measure
\[
d\mu(\xi)=\abss{\hat{f}(\xi)}^2/\normltwo{f}^2\,d\xi
\]
Then
\begin{align*}
\exp\big(-2s\,\,\calV(f)/\normltwo{f}^2\big) &= \exp\left(-2s\,\,\int\log\abs{\xi}\,d\mu(\xi)\right)\\
 &= \exp\left(\int\log\big(\abs{\xi}^{-2s}\big)\,d\mu(\xi)\right)\\
\text{(Jensen inequality)} &\le \int\abs{\xi}^{-2s}\,d\mu(\xi) = \frac{\normhms{f}^2}{\normltwo{f}^2}
\end{align*}
This proves $a)$. Combining it with Theorem~\ref{thm:mainthm} proves $b)$.
\end{proof}

\begin{proof}[Proof of Prop.~\ref{prop:exp_decay_geometric}]
Let $0<\eta<1$, then there exists a small $\rho>0$ such that for all $\xi\in\Rd$
\begin{equation} \label{eq:lower_bound_chi}
\abs{\hat{\chi}(\xi)} \ge \sqrt{\eta}\,\ind_{B_{\rho}(0)}(\xi)
\end{equation}
where $B_{\rho}(0)$ is the ball of radius $\rho$ centered at $0$. Indeed $\hat{\chi}(0)=\int\chi(x)\,dx=1$ (see Appendix for the definition of the Fourier transform) and $\hat{\chi}$ is continuous at $0$ (in fact $\hat{\chi}(\xi)=\frac{1}{\abss{B_1(0)}}J_{d/2}(2\pi\xi)/\abs{\xi}^{d/2}$ where $J_{\nu}$ is the classical Bessel function of order $\nu$, see for instance~\cite[Sec.~B.4]{grafakosclassicalbook}).

Let $B>1$, consider a Schwartz function $f$ such that $\calV(f)>0$ and let $0<\eps <\rho \exp\big(-B\calV(f)/\normltwo{f}^2\big)$. We have 
\[
\normltwo{f\conv\chieps}^2 \le \normlinf{f\conv\chieps}\normlone{f}
\]
using Hölder's inequality followed by Young's inequality, since $\normlone{\chieps}=\normlone{\chi}=1$. On the other hand $\normltwo{f\conv\chieps}^2 = \int\abss{\hat{\chi}(\eps\xi)}^2\,\abss{\hat{f}(\xi)}\,d\xi$ and using the previous lower bound~\eqref{eq:lower_bound_chi} on $\hat{\chi}$ we have
\begin{align*}
\normltwo{f\conv\chieps}^2  & \ge \eta\int_{\abs{\xi}\le \eps^{-1}\rho} \abss{\hat{f}(\xi)}^2\,d\xi\\
 & \ge \eta\left(\normltwo{f}^2 - \int_{\abs{\xi}>\eps^{-1}\rho} \abss{\hat{f}(\xi)}^2\,d\xi\right)
\end{align*}
Because of the way $\eps$ was chosen, $\set{\abs{\xi} > \eps^{-1}\rho} \subset \set{\abs{\xi} > \exp\big(B\calV(f)/\normltwo{f}^2\big)}$. Consequently we can bound the total mass of the high frequencies
\begin{align*}
\int_{\abs{\xi}>\eps^{-1}\rho} \abss{\hat{f}(\xi)}^2\,d\xi & \le \int_{\abs{\xi} > \exp(B\calV(f)/\normltwo{f}^2)} \abss{\hat{f}(\xi)}^2\,d\xi\\
 & \le \int_{\abs{\xi} > \exp(B\calV(f)/\normltwo{f}^2)} \frac{\log\abs{\xi}}{B\calV(f)/\normltwo{f}^2}\abss{\hat{f}(\xi)}^2\,d\xi \\
 & \le \frac{\normltwo{f}^2}{B}
\end{align*}
Note that we used $\calV(f)>0$. We deduce that
\[
\normltwo{f\conv\chieps}^2 \ge \eta \, (1-1/B)\normltwo{f}^2
\]
which implies
\[
\frac{\normlinf{f\conv\chieps}}{\normlinf{f}} \ge \eta \,\,  (1-1/B) \frac{\normltwo{f}^2}{\normlone{f}\normlinf{f}}
\]
Choosing at the beginning $\eta > 1-\kappa$ and $B$ big enough such that $\eta \,(1-1/B) > 1-\kappa$ and finally $A=\max\{B,\rho^{-1}\}$ proves part \emph{a}). 

Let us now prove part \emph{b}). On the torus $\Tor^d$, if $\th_0$ only has values $\pm 1$ then all the $L^q$ norms of $\th(t,\cdot)$ are equal to $1$. Thus part~\emph{a}) implies that if $\eps< A^{-1}\exp\big(-A\calV(\th(t,\cdot))\big)$ then 
\[
\frac{\normlinf{\th(t,\cdot)\conv\chieps}}{\normlinf{\th(t,\cdot)}} > 1-\kappa
\]
Thus, the definition of the geometric mixing scale $\eps(t)$ implies that
\[
\eps(t) \ge  A^{-1}\exp\big(-A\,\calV(\th(t,\cdot))\big)
\]
Combining this last bound with Theorem~\ref{thm:mainthm} proves~\emph{b}).
\end{proof}

\subsection{Blowup of positive fractional Sobolev norm} \label{sect:blowup_sobolev}
In this section we assume that the velocity field $u$ is bounded in $\dot{W}^{1,2}$ uniformly in time. By incompressibility the $L^2$ norm of $\th(t,\cdot)$ is constant in time
\[
\normltwo{\th(t,\cdot)} = \normltwo{\th_0}
\]
On the other hand, in~\cite{acm} the authors construct a solution $\th(t,\cdot)$ whose $\dot{H}^s$ norm blows up. More precisely, the initial value $\th_0$ is in $C^{\infty}_c(\Rd)$ and the solution $\th(t,\cdot)$ does not belong to $\dot{H}^s(\Rd)$ for any $s>0$ and $t>0$. 

Theorems~\ref{thm:mainthm} and~\ref{thm:W} give us insight on the blow up of positive fractional Sobolev norms by providing intermediate results between the conservation of the $L^2$ norm and the (possible) blow up of the $\dot{H}^s$ norm:

If $u$ is bounded in $\dot{W}^{1,2}$ uniformly in time, there exist various constants $C>0$ (denoted by the same letter for readability) such that for all $t\ge 0$
\begin{align}
&\int\abss{\hat{\th}(t,\xi)}^2\,d\xi  = C\\
&\int\log\abs{\xi}\,\abss{\hat{\th}(t,\xi)}^2\,d\xi \le C (1+t)\\
&\int(\log\abs{\xi})^2\,\abss{\hat{\th}(t,\xi)}^2\,d\xi \le C (1+t)^2\\
&\int\abs{\xi}^{2s}\,\abss{\hat{\th}(t,\xi)}^2\,d\xi\text{ can blow up, for any }s>0
\end{align}

We can add one more item to the list by recalling from Section~\ref{sec:functionals} that for small $s$,
\[
\norm{f}_{\dot{H}^s}^2 = \normltwo{f}^2 + 2s\,\calV(f) + 2s^2\,\calW(f) + O(s^3)
\]
Evidently, the first terms of the expansion of $\norm{\th(t,\cdot)}_{\dot{H}^s}^2$ as $s\to 0$ do not blow up.

\subsection{Linear growth of $\calV\big(\th(t,\cdot)\big)$ is sharp}
In this section we show that the linear growth of $\calV\big(\th(t,\cdot)\big)$ in Theorem~\ref{thm:mainthm} is sharp, i.e. there is an initial distribution $\th_0$ and a velocity field $u$ bounded in $\dot{W}^{1,p}$ uniformly in time such that $\calV\big(\th(t,\cdot)\big)$ grows linearly. To show this we use the results in~\cite{acm}, where the authors prove that the exponential lower bound on the functional mixing scale $\normhmo{\th(t,\cdot)}$ is sharp. We have

\begin{prop} \label{prop:linear_growth_V}
On the two-dimensional torus $\Tor^2$, for any $p>1$ there exists a velocity field $u$ bounded in $\dot{W}^{1,p}$ uniformly in time and a solution $\th$ to~\eqref{eq:conservation_of_mass} such that
\[
\text{For } n\le t< n+1,\quad\calV\big(\th(t,\cdot)\big) = \calV\big(\th(t-n,\cdot)\big) + n\log(\lambda^{-1})\normltwo{\th_0-\bar{\th}_0}^2
\]
where $\bar{\th}_0=\int\th_0(x)\,dx$ is the average of $\th_0$. Thus $\calV\big(\th(t,\cdot)\big)$ grows linearly
\[
\calV\big(\th(t,\cdot)\big) \ge m + (t-1)\log(\lambda^{-1}) \normltwo{\th_0-\bar{\th}_0}^2 \quad\text{with }m=\inf_{0\le t\le 1}\calV\big(\th(t,\cdot)\big)
\]
\end{prop}

\begin{remark}
Here we consider the dynamic on the torus $\Tor^2$, like in Prop.~\ref{prop:exp_decay_geometric}. To be rigorous we then need to consider the functional $\widetilde{\calV}$ and not $\calV$ (see Remark~\ref{rem:torus}). 
\end{remark}

\begin{proof}[Proof of Prop.~\ref{prop:linear_growth_V}]
We use the following result from~\cite{acm} 
\begin{fact*}[{Alberti, Crippa, Mazzucato '14~\cite[Sect~.1]{acm}}]
On the two-dimensional torus $\Tor^2$, there exists a velocity field $u$ bounded in $\dot{W}^{1,p}(\Tor^2)$ and a solution $\th$ to~\eqref{eq:conservation_of_mass} such that for all integers $n \ge 0$, 
\[
\text{If } n\le t<n+1, \quad \th(t,x) = \th\left(t-n,\frac{x}{\lambda^n}\right)
\]
where $1/\lambda$ is a positive integer.
\end{fact*}

Then for $n\le t<n+1$ and $k\in\Integ^d$ we have 
\[
\hat{\th}(t,k) = \begin{cases} \hat{\th}(t-n,\lambda^n k)\quad\text{if}\,k\in\lambda^{-n}\Integ^2\\
							   0\quad\text{otherwise}
			 	 \end{cases}
\]
thus, defining $\widetilde{\calV}$ like in Remark~\ref{rem:torus}:
\[
\widetilde{\calV}(f) = \sum_{k\in\Integ^d\setminus\{0\}} \log\abs{k}\,\abss{\hat{f}(k)}^2
\]
we have for $n\le t<n+1$
\[
\widetilde{\calV}\big(\th(t,\cdot)\big) = \widetilde{\calV}\big(\th(t-n,\cdot)\big) + n\log(\lambda^{-1})\left(\int_{\Tor^2}\th(t-n,x)^2\,dx - \left(\int_{\Tor^2} \th(t-n,x)\,dx\right)^2\right)
\]
which concludes the proof.
\end{proof}

\section*{Acknowledgements}
The author would like to thank Nader Masmoudi for suggesting the problem, Pierre Germain for offering helpful comments and Robert Kohn for his invaluable help in organizing the present work. Support is gratefully acknowledged from NSF grants DMS-1211806 and DMS-1311833.

\section{Appendix}

\subsection{Fourier Transform of $log$}

We use the following Fourier transform: for functions $f$ in the Schwarz class $\Schwartz(\Rd)$ the Fourier transform $\FT(f)$ or $\hat{f}$ of $f$ is
\[
\hat{f}(\xi) = \int_{\Rd}e^{-2i\pi\xi\cdot x}f(x)\,dx
\]

Let us now define the tempered distribution $T\in\Schwartz'(\Rd)$ by
\[
\bracket{T,f} = \int_{\abs{x}\le 1}\frac{f(x)-f(0)}{\abs{x}^d}\,dx + \int_{\abs{x}>1}\frac{f(x)}{\abs{x}^d}\,dx
\]
for any $f\in\Schwartz(\Rd)$. 
\begin{prop}
The Fourier transform of $T$ is
\[
\bracket{\hat{T},\psi} = \int\big(\zeta_d-\sigma_{d-1}\log\abs{\xi}\big)\,\psi(\xi)\,d\xi
\]
for any $\psi\in\Schwartz(\Rd)$, where $\sigma_{d-1}$ is the surface area of the unit sphere in $\Rd$ $\Sph^{d-1}$ and $\zeta_d$ is a constant.
\end{prop}

\begin{proof}
We adapt the proof of~\cite[§9,8.(d)]{vladimirov}.
Let $\psi$ be a test function in $\Schwartz$. Then
\begin{align*}
\bracket{\hat{T},\psi} = \bracket{T,\hat{\psi}} &= \int_{\abs{x}\le 1} \frac{1}{\abs{x}^d} \, \big(\hat\psi(x)-\hat\psi(0)\big)\,dx + \int_{\abs{x}>1}\frac{1}{\abs{x}^d}\,\hat{\psi}(x)\,dx\\
 & = \int_{\abs{x}\le 1} \frac{1}{\abs{x}^d}\int_{\xi}\big(e^{-2i\pi x\cdot\xi}-1\big)\psi(\xi)\,d\xi\,dx + \int_{\abs{x}>1} \frac{1}{\abs{x}^d}\int_{\xi}e^{-2i\pi x\cdot\xi}\,\psi(\xi)\,d\xi\,dx \\
 &= A+B
\end{align*}

To compute further we write the $x$ integral in spherical coordinates $x=r\om$, with $r>0$ and $\om\in\Sph^{d-1}$ the unit sphere in $\Real^d$. Then the first term can be written
\begin{align*}
A &= \int_{r=0}^1\int_{\om\in\Sph^{d-1}}\frac{1}{r^d} \int_{\xi} \big(e^{-2i\pi x\cdot\xi}-1\big)\psi(\xi)\,d\xi\,d\sigma(\om)\,r^{d-1}\,dr\\
  &= \int_{r=0}^1\frac{1}{r}\int_{\xi} \int_{\om}\big(e^{-2i\pi \xi\cdot r\om}-1\big)\,\psi(\xi)\,d\sigma(\om)\,d\xi\,dr
\end{align*}
where we have swapped the integrals in $\om$ and $\xi$. We have (see~\cite[Sec.~B.4.]{grafakosclassicalbook})
\[
\int_{\Sph^{d-1}}e^{-2i\pi \xi\cdot r\om}\,d\sigma(\om) = (2\pi)^{d/2} \tilde{J}_{\frac{d}{2}-1}(2\pi r\abs{\xi})
\]
where $\tilde{J}_{\nu}(s)=s^{-\nu}J_{\nu}(s)$ and $J_{\nu}$ is the classical Bessel function of order $\nu$. 

Continuing the computation,
\begin{align*}
A &= \int_0^1\frac{1}{r}\int\psi(\xi)\,\Big((2\pi)^{d/2} \tilde{J}_{\frac{d}{2}-1}(2\pi r\abs{\xi})-\sigma_{d-1}\Big)\,d\xi\,dr\\
  &= \int\psi(\xi)\int_0^{\abs{\xi}}\Big((2\pi)^{d/2} \tilde{J}_{\frac{d}{2}-1}(2\pi s)-\sigma_{d-1}\Big)\,\frac{ds}{s}\,d\xi
\end{align*}
where we swapped the integrals in $r$ and $\xi$ and did a change of variables in the $r$ integral $s=r\abs{\xi}$.

Similarly for the second term
\[
B = \int\psi(\xi)\int_{\abs{\xi}}^{\infty}(2\pi)^{d/2} \tilde{J}_{\frac{d}{2}-1}(2\pi s)\,\frac{ds}{s}\,d\xi
\]

Putting together $A$ and $B$ we get 
\begin{equation*}
A+B = \int\psi(\xi)\big(\zeta_d - \sigma_{d-1}\log\abs{\xi}\big)\,d\xi
\end{equation*}
with
\[
\zeta_d = \int_0^1\Big((2\pi)^{d/2}\tilde{J}_{\frac{d}{2}-1}(2\pi s) - \sigma_{d-1}\Big)\,\frac{ds}{s} + \int_1^{\infty}(2\pi)^{d/2}\tilde{J}_{\frac{d}{2}-1}(2\pi s)\,\frac{ds}{s}
\]
\end{proof}

\bibliographystyle{plain}
\bibliography{../biblio/bib/general_mixing,../biblio/bib/calderon_commutators}

\end{document}